\documentclass[11pt, letterpaper, english]{smfart}
\usepackage{smfthm}
\usepackage{amssymb}
\usepackage{tikz-cd}
\usepackage{mathtools}
\usepackage{hyperref}
\usepackage{enumitem}
\usepackage{comment}

\renewcommand{\geq}{\geqslant}
\renewcommand{\leq}{\leqslant}

\theoremstyle{plain}
\newtheorem{theorem}{Theorem}[section]

\newtheorem{lemma}[theorem]{Lemma}

\newtheorem{corollary}[theorem]{Corollary}

\theoremstyle{remark}
\newtheorem{remark}[theorem]{Remark}

\renewcommand{\AA}{\ensuremath{\mathbb{A}}}
\newcommand{\CC}{\ensuremath{\mathbb{C}}}

\newcommand{\QQ}{\ensuremath{\mathbb{Q}}}

\newcommand{\ZZ}{\ensuremath{\mathbb{Z}}}

\DeclareMathOperator{\Spec}{Spec}

\DeclareMathOperator{\h}{\mathsf{h}} 
\DeclareMathOperator{\NS}{\mathrm{NS}}

\title[{Algebraic cycles on varieties of generalized Kummer type}]{Algebraic cycles on hyper-K\"ahler varieties of generalized Kummer type}
\author{Salvatore Floccari}
\address{Institute of Algebraic Geometry, Leibniz University Hannover, Germany}
\address{\textit{Current address :} Fakult\"at f\"ur Mathematik, Universit\"at Bielefeld, Germany} 
\email{sfloccari@math.uni-bielefeld.de}

\author{Mauro Varesco}
\address{Mathematisches Institut, Universit\"at Bonn, Endenicher Allee 60, 53115 Bonn, Germany}
\email{varesco@math.uni-bonn.de}

\begin{document}
	
	\begin{abstract}
		We prove the conjectures of Hodge and Tate for any four-dimensional hyper-K\"ahler variety of generalized Kummer type. For an arbitrary variety $X$ of generalized Kummer type, we show that all Hodge classes in the subalgebra of the rational cohomology generated by $H^2(X,\QQ)$ are algebraic.
	\end{abstract}

	\maketitle
	
	\section{Introduction}
	Despite many efforts, the Hodge conjecture remains widely open. Much work has been devoted to the study of algebraic cycles on abelian varieties. In spite of several positive results, already in this setting the Hodge conjecture proved to be a rather formidable problem. It is in general open for abelian varieties of dimension at least $4$ (see \cite{moonenZahrinHodge, moonenZarhin}).
	
	Another interesting class of varieties with a trivial canonical bundle is that of hyper-K\"ahler varieties. See the articles \cite{beauville1983varietes} and \cite{Huy99} for general information. In the present paper we study the Hodge conjecture for hyper-K\"ahler varieties of generalized Kummer type ($\mathrm{Kum}^n$-varieties for short). By definition, these are deformations of Beauville's generalized Kummer varieties~(\cite{beauville1983varietes}) constructed from abelian surfaces.
	
	Important progress on $\mathrm{Kum}^n$-varieties came from the works of O'Grady~\cite{O'G21} and Markman \cite{markman2019monodromy}. They uncovered the relation between a $\mathrm{Kum}^n$-variety~$X$ and its intermediate Jacobian $J^3(X)$, which is shown to be an abelian fourfold of Weil type. Their results lead to a Torelli theorem for $\mathrm{Kum}^n$-varieties in terms of the Hodge structure on $H^3(X,\ZZ)$.
	Markman further constructs an algebraic cycle on $X\times J^3(X)$ realizing the canonical isomorphism $H^3(X,\QQ)\cong  H^1(J^3(X),\QQ)(-1)$ of Hodge structures, for any variety $X$ of $\mathrm{Kum}^n$-type. It follows that the Kuga-Satake correspondence (\cite{KUGA1967}) is algebraic for these varieties, as shown by Voisin \cite{voisinfootnotes}. See also \cite{VV} and \cite{vangeemenSpinor} for an account of these results.
	
	In the recent article \cite{floccariHCKum3}, the first author proved the Hodge conjecture for any six-dimensional variety~$K$ of~$\mathrm{Kum}^3$-type. No locally complete family of hyper-K\"ahler or abelian varieties of dimension at least $4$ satisfying the Hodge conjecture was previously known. 
	An important ingredient in the proof is the construction from \cite{floccariKum3} of a~K3 surface~$S_K$ naturally associated with $K$, and the fact that the Hodge conjecture holds for any power of this K3 surface. This was proven in \cite[Corollary 5.8]{floccariKum3}, using a theorem of the second author~\cite{varesco} and the aforementioned results of O'Grady, Markman and Voisin.
		
	In this article we obtain some results on the Hodge conjecture for varieties of generalized Kummer type of arbitrary dimension.
	
	\begin{theorem}\label{thm:main}
		Let $X$ be a projective manifold of $\mathrm{Kum}^n$-type, $n\geq 2$. Denote by $A_2^{\bullet}(X) \subset H^{\bullet}(X,\QQ)$ the subalgebra of the rational cohomology generated by $H^2(X,\QQ)$. Then any class in $ A_2^{2j}(X) \cap H^{j,j}(X)$ is algebraic, for any~$j$. 
	\end{theorem}

	To prove this theorem we use the works of Foster \cite{foster} and of the second author~\cite{varesco2023hodge} to show that an arbitrary variety $X$ of $\mathrm{Kum}^n$-type is related via an algebraic correspondence to the K3 surface~$S_K$ associated to some~$\mathrm{Kum}^3$-variety~$K$. Theorem \ref{thm:main} is then deduced from the Hodge conjecture for the powers of~$S_K$.  Our proof leads to the expectation that any variety of $\mathrm{Kum}^n$-type is naturally associated with a K3 surface, generalizing the construction given in \cite{floccariKum3} for the six-dimensional case.
	We remark that it should be possible 
	to obtain the theorem via the representation-theoretic methods of \cite{varesco}.
		
	For $j\leq n$, cup-product induces an isomorphism $A_2^{2j}(X)\cong \mathrm{Sym}^{j}(H^2(X,\QQ))$ of Hodge structures, by a theorem of Verbitsky \cite{verbitsky1996cohomology}.
	However, Theorem~\ref{thm:main} is not sufficient to prove the Hodge conjecture for $X$ as not all Hodge classes lie in $A^{\bullet}_2(X)$ (see \cite{green2019llv}). 
	For~$n=3$, the full Hodge conjecture proven in~\cite{floccariHCKum3}  is a stronger result and requires considerably more work. 
	For $n=2$, the Hodge classes in the complement of~$A_2^{\bullet}(X)$ form an $80$-dimensional subspace of the middle cohomology; Hassett and Tschinkel have shown in \cite{hassettTschinkel} that these classes are algebraic, for any~$X$ of $\mathrm{Kum}^2$-type.
	Hence, Theorem~\ref{thm:main} yields the following.
	\begin{corollary}\label{cor:Kum2}
		Let $X$ be a projective manifold of $\mathrm{Kum}^2$-type. Then the Hodge conjecture holds for $X$, i.e., $H^{j,j}(X)\cap H^{2j}(X,\QQ)$ consists of algebraic classes for any $j$.
	\end{corollary}
	
	Let now $k$ be a finitely generated field of characteristic $0$, with algebraic closure $\bar{k}$, and let $X/k$ be a smooth and projective variety over $k$. Given a prime number $\ell$, the absolute Galois group of $k$ acts on the $\ell$-adic \'etale cohomology of $X_{\bar{k}}$. In analogy with the Hodge conjecture, the Tate conjecture predicts that the subspace of Galois invariants in~$H_{\text{\'et}}^{2j}(X_{\bar{k}}, \QQ_{\ell}(j))$ is spanned by the fundamental classes of $k$-subvarieties of~$X$. See \cite{totaro} for general information on the Tate conjecture.
	
	\begin{corollary}\label{cor:TCKum2}
		Let $k\subset \CC$ be a finitely generated field and let $X/k$ be a smooth and projective variety such that $X_{\CC}$ is of $\mathrm{Kum}^2$-type. Then, for any prime number $\ell$, the strong Tate conjecture holds for~$X$, i.e., the Galois representations $H_{\text{\'et}}^{j}(X_{\bar{k}}, \QQ_{\ell})$ are semisimple and the subspace of Galois invariants in $H^{2j}_{\text{\'et}}(X_{\bar{k}}, \QQ_{\ell}(j))$ is the $\QQ_{\ell}$-span of fundamental classes of $k$-subvarieties of~$X$, for any $j$.
	\end{corollary}
	
	A third conjecture, the Mumford-Tate conjecture, connects those of Hodge and Tate; see~\cite[\S2.1]{moonen2017} for its statement. While this is a hard open problem in itself, the Mumford-Tate conjecture has been proven for any hyper-K\"ahler variety of known deformation type in~\cite{floccari2019}, \cite{soldatenkov19} and \cite{FFZ}. 
	As a consequence, the conjectures of Hodge and Tate are equivalent for such a variety; in particular, Corollary \ref{cor:TCKum2} is in fact equivalent to Corollary \ref{cor:Kum2}.
	
	\subsection*{Aknowledgements}
	We are grateful to Stefan Schreieder for many suggestions and remarks. We thank Ben Moonen for his comments on a first version of this draft, and the anonymous referee for his remarks. The first named author was partially funded by the Deutsche Forschungsgemeinschaft (DFG, German Research Foundation) – Project-ID 491392403 – TRR 358. The second named author was funded by ERC Synergy Grant HyperK, Grant agreement ID 854361.	
	
	\section{Motives}
	Grothendieck's theory of motives provides a useful framework to study the Hodge conjecture. 
	We will work with the category of homological motives with rational coefficients over~$\CC$, which we denote by~$\mathsf{Mot}$; see \cite{scholl1994classical} or \cite{andre} for its construction. The objects of $\mathsf{Mot}$ are triples $(X, p, n)$ where~$X$ is a smooth and projective complex variety, $p$ is an idempotent correspondence given by an algebraic class in $H^{2\dim X} (X\times X,\QQ)$, and $n$ is an integer. Morphisms are given by algebraic cycles modulo homological equivalence via the formalism of correspondences; more precisely, morphisms from $(X,p,n)$ to $(Y,q,m)$ are by definition the algebraic classes $\gamma\in H^{2\dim X - 2n + 2m}(X\times Y,\QQ)$ such that $\gamma\circ p = q\circ \gamma$.
	
	The category $\mathsf{Mot}$ is a pseudo-abelian tensor category.  The unit object for the tensor product is denoted by $\mathsf{Q}$; the Tate motives (resp., the Tate twists of a motive $\mathsf{m}$) will be denoted by $\mathsf{Q}(i)$ (resp., by $\mathsf{m}(i)$). Given a motive $\mathsf{m}$, we let $\langle \mathsf{m}\rangle_{\mathsf{Mot}}$ be the pseudo-abelian tensor subcategory of $\mathsf{Mot}$ generated by~$\mathsf{m}$, i.e., the smallest thick and full such subcategory containing $\mathsf{m}$ and closed under direct sums, tensor products, duals and subobjects.
	
	There is a natural contravariant functor $\h\colon \mathrm{SmProj}_\CC\to \mathsf{Mot}$, associating to a variety~$X$ its motive $\h(X)\coloneqq (X,\Delta,0)$.
	Here, $\Delta$ is the cohomology class of the diagonal in~$X\times X$.
	
	\begin{remark}
		Let $X$ be a smooth and projective variety. A polarization on $X$ gives a split inclusion of $\mathsf{Q}(-1)$ into $\h(X)$. It follows that $\langle \h(X) \rangle_{\mathsf{Mot}}$ contains all Tate motives. This category consists of the motives $ (Y, q, m)$ such that $Y$ is a power of $X$ or $\Spec(\CC)$. 
	\end{remark}
	
	The functor associating to $X$ its Hodge structure $H^{\bullet}(X,\QQ)$ factors as the composition of~$\h$ and the realization functor $R\colon \mathsf{Mot}\to \mathsf{HS}$ to the category $\mathsf{HS}$ of polarizable $\QQ$-Hodge structures. By definition, $R(X,p,n) \coloneqq p_*(H^{\bullet}(X,\QQ)(n))$, and $R$ is faithful.
	The Hodge conjecture is equivalent to the fullness of the realization functor $R$.
	
	\begin{remark}
		Let $X$ be a smooth and projective variety. Then the Hodge conjecture holds for $X$ and all of its powers if and only if the restriction of $R$ to $\langle \h(X)\rangle_{\mathsf{Mot}}$ is full.
	\end{remark} 
	
	Grothendieck's standard conjectures \cite{grothendieck1969standard} would ensure that the category $\mathsf{Mot}$ has much better properties than just being a pseudo-abelian tensor category. We recall the following theorem due to Jannsen \cite{Jan92} and Andr\'e~\cite{andre1996Motives}; see \cite[Theorem 4.1]{Ara06}.
	\begin{theorem}
		Let $X$ be a smooth and projective complex variety. Then the following are equivalent:
		\begin{itemize} 
			\item Grothendieck's standard conjectures hold for $X$;
			\item $\langle \h(X)\rangle_{\mathsf{Mot}}$ is a semisimple abelian category.
		\end{itemize}
	\end{theorem}
	
	In light of the above theorem, we will say that the standard conjectures hold for a motive $\mathsf{m}\in\mathsf{Mot}$ if the category $\langle \mathsf{m} \rangle_{\mathsf{Mot}}$ is semisimple and abelian.
	
	\begin{remark}\label{rmk:standard1}
	The standard conjectures are known to hold for curves, surfaces, and abelian varieties \cite{kleiman}. If they hold for $X$ and $Y$, then they hold for $X\times Y$ as well.
	\end{remark} 
	We recall below some known consequences of the standard conjectures.
	
	\begin{remark}\label{rmk:standard2}
		Assume that the standard conjectures hold for $\mathsf{m}\in \mathsf{Mot}$. Then, by \cite[Corollaire 5.1.3.3]{andre}, the restriction of the realization functor $R$ to $\langle \mathsf{m}\rangle_{\mathsf{Mot}}$ is conservative, which means that a morphism $f$ in this category is an isomorphism if and only if its realization $R(f)$ is an isomorphism of Hodge structures.
		Moreover, $\mathsf{m}$ (as well as any object in $\langle \mathsf{m}\rangle_{\mathsf{Mot}}$) admits a canonical weight decomposition $\mathsf{m}=\bigoplus_i \mathsf{m}^i$ such that $R(\mathsf{m}^i)$ is a pure Hodge structure of weight $i$; see \cite[\S5.1.2]{andre}. If the standard conjectures hold for the smooth and projective variety $X$, we shall thus write $\h(X)=\bigoplus_i \h^i(X)$ for the K\"unneth decomposition of $\h(X)$.
	\end{remark}
	
	We will also use the following easy fact from the theory of motives.
	\begin{remark}
		Assume that $\Gamma$ is a finite group acting on a smooth and projective variety $X$. Then the $\Gamma$-invariant part $\h(X)^{\Gamma}\coloneqq (X, p^\Gamma, 0)$ is a direct summand of the motive of $X$, cut out by the projector $p^{\Gamma}\coloneqq \frac{1}{|\Gamma|} \sum_{\gamma\in\Gamma} [\mathrm{graph}(\gamma)] \in H^{2\dim X}(X\times X,\QQ)$. Moreover, if the quotient $X/\Gamma$ is smooth, then $\h(X)^{\Gamma}$ equals the motive $\h(X/\Gamma)$.
	\end{remark}

	\section{Some recent results}
	Let $X$ be a $\mathrm{Kum}^n$-variety, $n\geq 2$. The automorphisms of $X$ which act trivially on its second and third cohomology groups form a group $\Gamma_n\cong (\frac{\ZZ}{(n+1)\ZZ})^4$, by \cite{boissiere2011higher} and \cite[Proposition 3.1]{hassettTschinkel}. 
	In \cite[\S12.5]{markman2019monodromy}, Markman constructs a four-dimensional abelian variety $T_X$ associated to $X$, isogenous to the intermediate Jacobian~$J^3(X)$. 
	He further shows that $\Gamma_n$ acts on~$T_X$ via translations and that the quotient $M_X\coloneqq (X\times T_X)/\Gamma_n$ by the anti-diagonal action is a smooth holomorphic symplectic variety deformation equivalent to a smooth and projective moduli space of stable sheaves on an abelian surface (\cite{huybrechts2010geometry, Yos01}). Building on Markman's results and the strategy used by Charles and Markman in \cite{CM13} to prove the standard conjectures for~$\mathrm{K}3^{[n]}$-varieties, Foster obtaines the following theorem. 
	
	\begin{theorem}[{\cite[Theorem 4.1]{foster}}] \label{thm:foster}
		Let $X$ be a variety of $\mathrm{Kum}^n$-type. Then the standard conjectures hold for $M_X$. 
	\end{theorem} 
	\begin{remark}\label{rmk:remark}
		The cohomology of $M_X$ is naturally identified with the $\Gamma_n$-invariants in $H^{\bullet}(X\times T_X,\QQ)$. Since $\Gamma_n$ acts trivially on the cohomology of $T_X$, we have 
		\[
		H^{\bullet}(M_X,\QQ) = 	H^{\bullet}(X,\QQ)^{\Gamma_n} \otimes H^{\bullet}(T_X,\QQ).
		\]
		As for any smooth projective complex variety, we can split off from $\h(T_X)$ a direct summand $\h^0(T_X)=\mathsf{Q}$ in $\mathsf{Mot}$, cut out by the cohomology class of the algebraic cycle $T_X \times \{t\}$ on $T_X \times T_X$, for any point $t\in T_X$. We conclude that 
		$$\h(M_X) = (\h(X) \otimes \h(T_X))^{\Gamma_n} $$ 
		contains the motive
		$\h(X)^{\Gamma_n}= (\h(X)\otimes \h^0(T_X))^{\Gamma_n}$ as a direct summand. Hence, by Foster's Theorem \ref{thm:foster}, the standard conjectures hold for~$\h(X)^{\Gamma_n}$.
		In particular, the motive $\h(X)^{\Gamma_n}$ admits a weight decomposition; recalling that $\Gamma_n$ acts trivially on the second cohomology of $X$, the component $\h^2(X)$ of $\h(X)$ is a well-defined direct summand, contained in $\h(X)^{\Gamma_n}$.
	\end{remark}
	
	The following result of the second author \cite{varesco2023hodge} is crucial for our proof of Theorem \ref{thm:main}. It allows us to relate varieties of generalized Kummer type of different dimensions via algebraic correspondences. Let~$r\neq 0$ be a rational number. 
	A similitude of multiplier~$r$ between two quadratic spaces $V_1$ and~$V_2$ is a linear isomorphism $t\colon V_1\to V_2$ which multiplies the form by a factor $r$, i.~e., such that~$(t(u),t(w))_2=r \cdot (u,w)_1$ for all $u,w\in V_1$.
	
	Recall (see \cite[\S1.9]{Huy99}) that the second cohomology of any hyper-K\"ahler variety $X$ is equipped with a non-degenerate symmetric bilinear form, called the Beauville-Bogomolov form, so that $H^2(X,\QQ)$ is naturally regarded as a quadratic space. We denote by $H^2_{\mathrm{tr}}(X,\QQ)$ its transcendental part, that is, the orthogonal complement to the N\'eron-Severi group $\NS(X) \otimes_{\ZZ} \QQ \subset H^2 (X,\QQ)$ with respect to the Beauville-Bogomolov form.
	
	\begin{theorem}[{\cite[Theorem 0.4]{varesco2023hodge}}] \label{thm:varesco}
		Let $X_1, X_2$ be varieties of generalized Kummer type (not necessarily of the same dimension). Assume that 
		$$\phi\colon H^2(X_1,\QQ) \xrightarrow{\ \sim \ } H^2(X_2,\QQ)$$
		is a rational Hodge similitude. 
		Then $\phi$ is induced by an algebraic cycle on $X_1\times X_2$.
	\end{theorem} 
	\begin{proof} 
		As $\phi$ is a morphism of Hodge structures, it is the sum of Hodge similitudes $\phi_{\mathrm{tr}}\colon H^2_{\mathrm{tr}}(X_1,\QQ) \xrightarrow{ \ \sim \ } H^2_{\mathrm{tr}}(X_2,\QQ)$ and $\phi_{\mathrm{alg}}\colon \NS(X_1)_{\QQ} \xrightarrow{ \ \sim \ } \NS(X_2)_{\QQ}$.
		Now \cite[Theorem~0.4]{varesco2023hodge} implies that $\phi_{\mathrm{tr}}\colon H^2_{\mathrm{tr}}(X_1,\QQ)\xrightarrow{\ \sim \ } H^2_{\mathrm{tr}}(X_2,\QQ)$ is induced by an algebraic cycle. Since $\NS(X_1)_{\QQ}$ and $\NS(X_2)_{\QQ}$ consist of algebraic classes, $\phi_{\mathrm{alg}}$ is necessarily induced by an algebraic cycle, and Theorem \ref{thm:varesco} follows. 
	\end{proof} 
	
	The main observation behind Theorem \ref{thm:varesco} is that the similitude $\phi$ induces an isogeny between the Kuga-Satake varieties of $X_1$ and $X_2$. The statement is then deduced using the algebraicity of the Kuga-Satake correspondence and Foster's Theorem \ref{thm:foster}.

	\section{Conclusion}
	
	As mentioned in the introduction, we will use the construction of a K3 surface $S_K$ associated to any variety $K$ of $\mathrm{Kum}^3$-type, given by the first author in \cite[Theorem 1.2]{floccariKum3}. 
	We will need the following result which is obtained from the construction.

	\begin{theorem}[\cite{floccariKum3}] \label{thm:floccari}
		Let $K$ be a variety of $\mathrm{Kum}^3$-type, with associated K3 surface $S_K$. There exists an algebraic cycle on $S_K\times K$ which induces a Hodge similitude of multiplier~$2$
		\[
		\psi\colon H^2_{\mathrm{tr}}(S_K,\QQ) \xrightarrow{\ \sim \ } H^2_{\mathrm{tr}}(K,\QQ).
		\]
		Moreover, the Hodge conjecture holds for any power of $S_K$.
	\end{theorem} 
	\begin{proof}
		We review the proof of this result. Let $K$ be any manifold of $\mathrm{Kum}^3$-type. Then, following \cite[\S2]{floccariKum3}, there exists an action of $G\coloneqq (\ZZ/2\ZZ)^5$ on $K$ such that the quotient $K/G$ admits a crepant resolution $Y_K\to K/G$, with $Y_K$ a hyper-K\"ahler manifold of $\mathrm{K}3^{[3]}$-type (this is \cite[Theorem 1.1]{floccariKum3}). 
		Let us assume from now on that $K$ is a projective variety of $\mathrm{Kum}^3$-type. Then, by \cite[Theorem~1.2]{floccariKum3}, the variety $Y_K$ is birational to a moduli space $M_{S_K, H}(v)$ of $H$-stable sheaves on a uniquely determined K3 surface $S_K$, for a polarization $H$ and a Mukai vector $v$ on $S_K$. This is the K3 surface $S_K$ associated to $K$ appearing in the above statement. 
		
		Now, by construction, we have a rational map $r\colon K\dashrightarrow Y_K$. By \cite[Lemma~4.8]{floccariKum3}, setting $\phi\coloneqq \frac{1}{16}r_*$ we obtain a Hodge similitude of multiplier $2$
		\[
		\phi\colon H^2_{\mathrm{tr}} (Y_K,\QQ)\xrightarrow{\ \sim \ } H^2_{\mathrm{tr}}(K,\QQ)
		\]
		between the transcendental Hodge structures. This map is induced by a multiple of the closure of the graph of the rational map $r$, and, hence, by an algebraic cycle on $Y_K\times K$.
		Moreover, by \cite[Lemma 2.6]{Huy99}, a birational map $M_{S_K, H}(v) \dashrightarrow Y_K$ yields a Hodge isometry $\phi' \colon H^2(M_{S_K, H}(v), \QQ) \xrightarrow{ \ \sim \ } H^2(Y_K,\QQ)$, induced by the closure of the graph and therefore algebraic. Finally, using the Mukai homomorphism (\cite{mukai1987moduli, O'G97}), one obtains that there exists a Hodge isometry 
		\[ \phi''\colon H^2_{\mathrm{tr}}(S_K,\QQ)\xrightarrow{ \ \sim \ } H^2_{\mathrm{tr}}(M_{S_K, H}(v),\QQ) \]
		induced by an algebraic cycle on $S_K \times M_{S_K, H}(v)$; see \cite[pp. 407]{floccariKum3} for more details. 
		
		We conclude that the composition $\psi\coloneqq \phi\circ \phi' \circ \phi''$ is a Hodge similitude of multiplier $2$
		\[
		\psi\colon H^2_{\mathrm{tr}}(S_K,\QQ)\xrightarrow{ \ \sim \ } H^2_{\mathrm{tr}}(K,\QQ),
		\]
		induced by an algebraic cycle on $S_K\times K$. This proves the first statement. The second statement is \cite[Corollary 5.8]{floccariKum3}. 		   
	\end{proof}
	
	Let $X$ be a $\mathrm{Kum}^n$-variety, $n\geq 2$, and denote by $A_2^{\bullet}(X)$ the subalgebra of the rational cohomology generated by $H^2(X,\QQ)$. 
	By Foster's Theorem \ref{thm:foster}, the degree~$2$ component~$\h^2(X)$ of $\h(X)$ is well-defined and it is  a direct summand of the~$\Gamma_n$-invariant part~$\h(X)^{\Gamma_n}$ of $\h(X)$, see Remark \ref{rmk:remark}. 
	
	\begin{lemma}\label{lem:uno}
		The subalgebra $A_2^{\bullet}(X)\subset H^{\bullet}(X,\QQ)$ is the realization of a submotive~$\mathsf{a}_2(X)$ of $\h(X)$.
		Moreover, $\mathsf{a}_2(X)\in \langle \h^2(X)\rangle_{\mathsf{Mot}}$.
	\end{lemma}
	\begin{proof}
		For $i>0$, we let $\delta_{i+1} \in H^{2i\dim X } (X^{i+1}, \QQ)$ denote the class of the small diagonal $\{(x, \dots, x)\} \subset X^{i+1}$. This class induces the cup-product morphism $\h(X)^{\otimes i} \to \h(X)$, which restricts to a morphism $\left(\h(X)^{\Gamma_n}\right)^{\otimes i} \to \h(X)^{\Gamma_n}$ in $\mathsf{Mot}$.
		Since $\h^2(X)$ is a direct summand of~$\h(X)^{\Gamma_n}$, the cup-product induces a morphism of motives
		\[
		\beta\colon \bigoplus_i \h^2(X)^{\otimes i} \to \h(X)^{\Gamma_n}.
		\]
		Note that $\beta$ is a morphism in the category $\langle \h(X)^{\Gamma_n}\rangle_{\mathsf{Mot}}$, which is abelian and semisimple by Foster's Theorem \ref{thm:foster} and Remark \ref{rmk:remark}. Therefore, the image of $\beta$ is a submotive $\mathsf{a}_2(X) \subset \h(X)$, whose realization is $A_2^{\bullet}(X) \subset H^{\bullet}(X,\QQ)$. By semisimplicity, $\mathsf{a}_2(X)$ is a direct summand of~$\bigoplus_i \h^2(X)^{\otimes i}$, and hence $\mathsf{a}_2(X)\in \langle \h^2(X)\rangle_{\mathsf{Mot}}$.
	\end{proof}
	
	The next lemma shows that $X$ is Hodge similar to a variety of $\mathrm{Kum}^3$-type.
	
	\begin{lemma}\label{lem:due}
		For any variety $X$ of $\mathrm{Kum}^n$-type, there exists a  variety $K$ of $\mathrm{Kum}^3$-type and a Hodge similitude of multiplier $n+1$
		\[
		\phi \colon H^2(K,\QQ) \xrightarrow{\ \sim \ } H^2(X,\QQ)
		\]
		with respect to the Beauville-Bogomolov pairings.
	\end{lemma}
	\begin{proof} 		
		By \cite{beauville1983varietes}, the integral second cohomology group of a $\mathrm{Kum}^n$-variety is identified with the lattice $\Lambda_{\mathrm{Kum}^n} = \mathrm{U}^{\oplus 3} \oplus \langle -2n-2\rangle$, where $\mathrm{U}$ is a hyperbolic plane. It is easy to define a rational similitude
		\[ 
		\phi_n \colon \Lambda_{\mathrm{Kum}^3}\otimes \QQ \xrightarrow{\ \sim \ } \Lambda_{\mathrm{Kum}^n}\otimes \QQ
		\]
		of multiplier $n+1$. Explicitly, let $e_i^n,f_i^n$, $i=1,2,3$ and $\xi^n$ be a basis of $\Lambda_{\mathrm{Kum}^n}$, where: $e^n_i,f^n_i$ are isotropic and $(e^n_i,f^n_i)=1$, the planes $\langle e^n_i,f^n_i\rangle$ and $\langle e^n_j, f^n_j\rangle$ are orthogonal for $i\neq j$, $\xi^n$ has square $-2n-2$ and it is orthogonal to each $e^n_i$ and $f^n_i$.
		Then $\phi_n$ is defined via $$e^3_i\mapsto e_i^n, \ \ \ f_i^3\mapsto (n+1) f_i^n, \ \text{ for } i=1,2,3, \ \ \ \xi^3\mapsto 2 \xi^n.$$ 
		
		Let $X$ be a $\mathrm{Kum}^n$-variety, and let $\eta\colon H^2(X,\ZZ)\xrightarrow{\ \sim  \ } \Lambda_{\mathrm{Kum}^n}$ be an isometry. 
		The Hodge structure on the left hand side is determined by its period $[\sigma]=\eta(H^{2,0}(X))$, which is an isotropic line in $\Lambda_{\mathrm{Kum}^n}\otimes \CC$ such that $(\sigma,\bar{\sigma})>0$. Via the similitude $\phi_n^{-1}$, we obtain the isotropic line $[\theta]\coloneqq [\phi_n^{-1}(\sigma)]$ in~$ \Lambda_{\mathrm{Kum}^3}\otimes \CC$, such that $(\theta, \bar{\theta})>0$. 
		By the surjectivity of the period map \cite[Theorem 8.1]{Huy99}, there exists a manifold $K$ of $\mathrm{Kum}^3$-type with period~$[\theta]$, which means that there is an isometry $\eta'\colon H^2(K,\ZZ)\xrightarrow{\ \sim \ } \Lambda_{\mathrm{Kum}^3}$ mapping~$H^{2,0}(K)$ to~$[\theta]$. 
		By construction, the composition $\eta^{-1}\circ \phi_n\circ \eta'$ gives an isomorphism of rational Hodge structures
		\[
		\phi \colon H^2(K,\QQ) \xrightarrow{\ \sim \ } H^2(X,\QQ),
		\]
		and hence $\phi$ is a rational Hodge similitude.	Since $X$ is projective, $K$ is projective as well, thanks to Huybrechts' projectivity criterion \cite[Theorem 3.11]{Huy99}. 
	\end{proof} 
	
	We can now complete the proofs of our main results.
	
	\begin{proof}[Proof of Theorem \ref{thm:main}]
		Given the projective $\mathrm{Kum}^n$-variety $X$, we consider the $\mathrm{Kum}^3$-variety $K$ given by Lemma \ref{lem:due}. The same lemma gives a Hodge similitude $\phi\colon H^2(K,\QQ)\xrightarrow{\ \sim \ } H^2(X,\QQ)$, which is induced by an algebraic cycle on $K\times X$ by Theorem \ref{thm:varesco}.
		Applying Theorem \ref{thm:floccari} to $K$, we obtain the associated K3 surface~$S_K$ and a Hodge similitude $\psi\colon H^2_{\mathrm{tr}}(S_K,\QQ) \xrightarrow{ \ \sim \ } H^2_{\mathrm{tr}}(K,\QQ)$ induced by an algebraic cycle on $S_K\times K$. 		
		The composition of $\psi$ with $\phi$ thus gives a rational Hodge similitude of multiplier $2n+2$
		\begin{equation*}\label{eq:HodgeIsometry}
			\Psi \colon H_{\mathrm{tr}}^2(S_K,\QQ) \xrightarrow{\ \sim \ } H_{\mathrm{tr}}^2(X,\QQ)
		\end{equation*}
		of transcendental lattices, which is induced by an algebraic cycle on $S_K\times X$.
	
		Consider the submotive $\h^2(X)$ of $\h(X)$. By Lefschetz (1,1) theorem, we have the decomposition $\h^2(X) = \h^2_{\mathrm{tr}}(X) \oplus \h^2_{\mathrm{alg}}(X)$ into transcendental and algebraic part (and similarly for $S_K$); the algebraic part is a sum of Tate motives $\mathsf{Q}(-1)$.
		Since $\Psi$ is induced by an algebraic cycle on $S_K \times K$, it is the realization of a morphism
		\[
		\widetilde{\Psi}\colon \h_{\mathrm{tr}}^2(S_K) \xrightarrow{\ \ \ } \h_{\mathrm{tr}}^2(X)
		\]
		of motives. 
		Recall from Remark \ref{rmk:remark} that $\h^2(X)$ is a direct summand of $\h(X)^{\Gamma_n}$, and in particular the motive $\h^2(X)$ belongs to $\langle \h(M_X)\rangle_{\mathsf{Mot}}$. Therefore, $\widetilde{\Psi}$ is a morphism in the subcategory~$\langle \h(S_K\times M_X)\rangle_{\mathsf{Mot}}$ of~$\mathsf{Mot}$. The standard conjectures hold for $S_K\times M_X$ by Theorem \ref{thm:foster} and Remark \ref{rmk:standard1}. By conservativity (see Remark~\ref{rmk:standard2}), it follows that $\widetilde{\Psi}$ is an isomorphism of motives, since its realization $\Psi$ is an isomorphism of Hodge structures. 
		As $\h^2(X)$ is the sum of $\h^2_{\mathrm{tr}}(X)$ and Tate motives, we conclude that $\h^2(X)\in \langle \h(S_K)\rangle_{\mathsf{Mot}}$. 
		
		Consider now the submotive $\mathsf{a}_2(X)\subset \h(X)$ constructed in Lemma \ref{lem:uno}. By the above,~$\mathsf{a}_2(X)$ belongs to~$\langle \h(S_K)\rangle_{\mathsf{Mot}}$.
		Since, by Theorem \ref{thm:floccari}, the Hodge conjecture holds for all powers of $S_K$, the realization functor $R$ is full when restricted to $\langle \h(S_K) \rangle_{\mathsf{Mot}}$, and we deduce that any Hodge class in $A_2^{\bullet}(X)\subset H^{\bullet}(X,\QQ)$ is algebraic.
	\end{proof}
	
	\begin{remark}
		Let $X$ be a $\mathrm{Kum}^n$-variety as above and consider any power $Z=X^r$. Denote by $A_2^{\bullet}(Z)$ the subalgebra of $H^{\bullet}(Z,\QQ)$ generated by $H^2(Z,\QQ)$. Then our argument implies that all Hodge classes in $A_2^{\bullet}(Z)$ are algebraic.             
		In fact, note that $A_2^{\bullet}(Z)$ is the graded tensor product $A_2^{\bullet}(X)^{\otimes r}$, because $H^1(X,\QQ)$ is zero. With notation as in the above proof, the argument given implies that $A_2^{\bullet}(Z)$ is the realization of a submotive $\mathsf{a}_2(Z)$ of~$\h(Z)$, and moreover that $\mathsf{a}_2(Z)=\mathsf{a}_2(X)^{\otimes r}$ belongs to $\langle \h(S_K) \rangle_{\mathsf{Mot}}$. As the Hodge conjecture holds for all powers of $S_K$, it follows that any Hodge class in $A_2^{\bullet}(Z)$ is algebraic.
	\end{remark}
	
	\begin{proof}[Proof of Corollary \ref{cor:Kum2}]
		When $X$ is of $\mathrm{Kum}^2$-type, the complement of $A_2^{\bullet}(X)$ in~$H^{\bullet}(X,\QQ)$ consists of the odd cohomology and of an $80$-dimensional space of Hodge classes in~$H^4(X,\QQ)$, by \cite[Example 4.6]{looijenga1997lie}. 
		The classes in this $80$-dimensional subspace of the middle cohomology remain Hodge on any deformation of~$X$, and Hassett and Tschinkel have shown in \cite[Theorem 4.4]{hassettTschinkel} that they are all algebraic. Together with our Theorem \ref{thm:main}, this implies that the Hodge conjecture holds for any $X$ of $\mathrm{Kum}^2$-type.
	\end{proof}
	\begin{proof}[Proof of Corollary \ref{cor:TCKum2}]
		As mentioned in the introduction, the Mumford-Tate conjecture has been proven for any hyper-K\"ahler variety $X/k$ of known deformation type, by work of the first author \cite{floccari2019}, Soldatenkov \cite{soldatenkov19} and of the first author with Fu and Zhang~\cite{FFZ}. The final result may be found in \cite[Theorem 1.18]{FFZ}. As a consequence, the Galois representations on $H^{j}_{\text{\'et}}(X_{\bar{k}},\QQ_{\ell})$ are semisimple, and the Tate conjecture for $X/k$ is equivalent to the Hodge conjecture for $X_{\mathbb{C}}$ (see \cite[Proposition 2.3.2]{moonen17}). Therefore, Corollary \ref{cor:TCKum2} follows from Corollary~\ref{cor:Kum2}.
	\end{proof} 
	
	
	\bibliographystyle{smfplain}
	\bibliography{bibliographyNoURL}{}
	
\end{document}